\documentclass[12pt]{amsart}
\usepackage{amssymb,latexsym, verbatim}
\usepackage{enumerate}

\newtheorem{corollary}{Corollary}[section]
\newtheorem*{main*}{Main~Theorem}

\newtheorem{conjecture}{Conjecture}
\newtheorem{thm}{Theorem}[section]

\newtheorem{lem}[thm]{Lemma}

\theoremstyle{definition}
\newtheorem{defin}[thm]{Definition}
\newtheorem{remark}[thm]{Remark}
\numberwithin{equation}{section}

\renewcommand{\i}{\infty}

\begin{document}

\title[The Logarithm of Irrational Numbers and Beatty Sequences]{The Logarithm of Irrational Numbers and Beatty Sequences}

\author[G. Polanco E.]{Geremias Polanco E}
\address{Math Department, School of Natural Science\\ Hampshire College\\
893 West St, Amherst, MA, 01002}
\email{gpeNS@hampshire.edu}

\date{April 2011}

\begin{abstract}
In this paper we find an identity that gives a representation for the logarithm of any two irrational numbers $a, b >1$ in terms of a series whose terms are ratios of elements from their Beatty Sequences. We also show that Sturmian sequences can be defined in terms of these ratios. Furthermore, we find an identity for such series that bears a superficial resemblance to (a discrete version of) Frullani's Integral.
\end{abstract}

\maketitle

\textbf{Notation:} Throughout this paper, we use $[x]$ to represent the greatest integer not exceeding $x$, and define $\{ x\}:=x-[x]$. Unless otherwise specified, the letters $a$ and $b$ represent irrational numbers. Also, $A$ and $B$ represent the Beatty sequence $A:=\left([ a n ] \right)_{n=1}^{\infty}$ and $B:=\left([ b n ] \right)_{n=1}^{\infty}$. We will often refer to the integers when we really mean the positive integers, but this will not generate any conceptual ambiguity.\\

\section{Introduction and Statement of Main Result}

Sturmian sequences are part of the general study of combinatorial properties of finite and infinite words that play a role in 
various fields of physics, mathematics, biology and computer science (see \cite{berstel1996} and the references therein). 
In \cite{obryant2002} O'Bryant obtains a special representation for a power series whose $n^{\text{th}}$ term is nonzero for $n$ in a 
particular Beatty sequence and zero otherwise. He then uses that representation to generalize the 
Rayleigh-Beatty Theorem. Other works  (see \cite{komatsu1996} and the references therein) also associate Beatty sequences 
with power series. In another direction, the asymptotic behavior of sums of arithmetic functions involving Beatty sequences and some of 
their applications have been studied in, e.g., \cite{abercrombie2009}, \cite{abercrombie1995}. In this paper we 
study the behavior of an infinite series involving Beatty sequences, establish its convergence and 
obtain (see Theorem \ref{thbrassthm2}) an identity involving the logarithm of their generating irrational slope. The terms of 
the series are differences of ratios of Beatty sequences, not necesarily complementary (such differences of ratios have
interesting properties, one of them being that Sturmian sequences can be defined in terms of them; see Section \ref{thbrasssec6}). 
We find partial sums of arithmetical functions defining Sturmian sequences,
 and partial sums of products of such arithmetical functions. These are then applied with other techniques in order to 
prove our main result.

\begin{main*}
\label{thbrassthm1a}
Let $a>1$ and $b>1$ be irrational numbers. Let $a_n=[ an]$
and $b_n=[ bn]$. Then
$$\lim_{k \to \infty}\frac{1}{k} \sum_{n=1}^\infty  \left(\frac{a_{n+k}}{a_n}-\frac{b_{n+k}}{b_n}\right) + \sum_{n=1}^{\infty}\frac{a \{a^{-1} (n+1)\}-b\{b^{-1}(n+1)\}}{n(n+1)}= \log\frac{a}{b} \text{.}$$
\end{main*}

Sections 2 through 8 of this paper are used to prove this theorem. In sections 2 through 6 we work with a particular case that is interesting because it treats the case when the sequences $A$ and $B$ are complementary. We also use these sections to introduce a series of lemmas  and notation that are used throughout the paper. Sections 7 and 8 are used to prove the main result. Specifically, the main theorem is proved as a corollary of Theorem \ref{thbrassthm2}.  The final section of this article is used to show how Beatty ratios can be used to define inhomogeneous Sturmian sequences and to make some final remarks about Frullani's Integral. We now concentrate in the limit expression on the left hand side of
the main theorem. For any fixed integer $k$, we have 
\begin{align}
\label{thbrass1a}
 &\frac{a_{n+k}}{a_n}-\frac{b_{n+k}}{b_n}=\frac{b_na_{n+k}-a_nb_{n+k}}{a_nb_n}=\\\nonumber
  &\frac{\left(bn-\{bn\}\right)\left(a(n+k)-\{a(n+k)\}\right)-\left(an-\{an\}\right)\left(b(n+k)-\{b(n+k)\}\right)}{a_nb_n}\nonumber
\end{align}
If we expand the numerator and add over $n$ we obtain
\begin{align}
\label{thbrass1a1}
\sum_{n=1}^\infty &=\left(\frac{a_{n+k}}{a_n}-\frac{b_{n+k}}{b_n}\right)=\\\nonumber
&\sum_{n=1}^\infty\frac{-bn\{a(n+k)\}-an\{bn\}-ak\{bn\}+\{bn\}\{a(n+k)\}}{a_nb_n}+\\\nonumber
 &\sum_{n=1}^\infty\frac{an\{b(n+k)\}+bn\{an\}+bk\{an\}-\{an\}\{b(n+k)\}}{a_nb_n}\text{.}
\end{align}
The two summands that are multiplied by $k$, upon divinding by $k$, yield 
\begin{equation}
 \label{thbrass1a2}
\sum_{n=1}^\infty \frac{a\{bn\}-b\{an\}}{a_nb_n}\text{.}
\end{equation}
Also, the summands with the product of two fractional parts, after dividing by $k$ give $O(1/k)$.
The remaining summands, after dividing by $k$, are
\begin{equation}
 \label{thbrass1a3}
\frac{1}{k}\sum_{n=1}^\infty \frac{an\{b(n+k)\}-an\{bn\}}{a_nb_n}+\frac{1}{k}\sum_{n=1}^\infty \frac{-bn\{a(n+k)\}+bn\{an\}}{a_nb_n}
\end{equation}
Consider now the first sum in \eqref{thbrass1a3} and note that the denominator can be written as $abn^2(1+O(1/n))$. 
Hence, for any fixed $M=M(k)$ to be specified later, we can write 
\begin{align}
 \label{thbrass1a4}
\sum_{n=1}^M \frac{an\{b(n+k)\}}{ka_nb_n}- &\sum_{n=1}^{M+k}\frac{an\{bn\}}{ka_nb_n}= \\\nonumber
&  \sum_{n=1}^M \frac{an\{b(n+k)\}}{abkn^2}-\sum_{n=1}^{M+k}\frac{an\{bn\}}{abkn^2}+O\left(\frac{1}{k}\right)=\\\nonumber
&\sum_{n=1}^M \frac{\{b(n+k)\}}{bkn}-\sum_{n=1}^{M+k}\frac{\{bn\}}{bkn}+O\left(\frac{1}{k}\right)=\\\nonumber
&\sum_{n=1}^{M} \left(\frac{\{b(n+k)\}}{kbn}-\frac{\{b(n+k)\}}{bk(n+k)}\right)+O\left(\frac{\log k}{k}\right)=\\\nonumber
&\sum_{n=1}^{M} \frac{\{b(n+k)\}}{kb}\left(\frac{1}{n}-\frac{1}{(n+k)}\right)+O\left(\frac{\log k}{k}\right)=\\\nonumber
& \frac{1}{kb}\sum_{n=1}^{M}\left(\frac{1}{n}-\frac{1}{(n+k)}\right)+O\left(\frac{\log k}{k}\right)=\\\nonumber
&O\left(\frac{\log(M+ k)}{k}\right)+O\left(\frac{\log k}{k}\right)\text{.}\nonumber
\end{align}
\noindent The last equality holds because the only surviving terms inside the telescoping sum in the previous step are $1/n$ for $n=1,...,k$ and
$-1/n$ for $n=M+1,...,M+k$. Taking $M=k^2$, we see that the left hand side of \eqref{thbrass1a4} yields
\begin{equation}
 \label{thbrass1a5}
\sum_{n=1}^M \frac{an\{b(n+k)\}}{ka_nb_n}-\sum_{n=1}^{M+k}\frac{an\{bn\}}{ka_nb_n}=O\left(\frac{\log k}{k}\right)\text{,}
\end{equation}
\noindent which approaches zero as $k$ goes to infinity. Hence we see that when taking limit as $k$ approaches infinity, \eqref{thbrass1a3} equals zero.
Thus, if we divide \eqref{thbrass1a1} by $k$ and take limit as $k$ goes to infinity, we have
\begin{equation}
 \label{thbrass1a6}
\lim_{k\to\infty}\frac{1}{k}\sum_{n=1}^\infty \left(\frac{a_{n+k}}{a_n}-\frac{b_{n+k}}{b_n}\right)= \sum_{n=1}^\infty \frac{a\{bn\}-b\{an\}}{a_nb_n}\text{.}
\end{equation}
Putting together \eqref{thbrass1a6} and Theorem \ref{thbrassthm1a} we have the following corollary.

\begin{corollary}
 \label{thbrasscor1a}
If $a>1$ and $b>1$ are irrational numbers, then
  $$\sum_{n=1}^\infty \left(\frac{a\{bn\}-b\{an\}}{a_nb_n}+\frac{a \{a^{-1} (n+1)\}-b\{b^{-1}(n+1)\}}{n(n+1)}\right)=\log\frac{a}{b}\text{.}$$
\end{corollary}

Now, it is reasonable to expect that, in absolute value, the right hand side of \eqref{thbrass1a6} is approximately $\displaystyle \left|\frac{\pi^2}{6b}-\frac{\pi^2}{6a}\right|$.
Thus, since both $a$ and $b$ are greater than $1$, it is reasonable to conjecture the following:
\begin{conjecture}
 \label{thbrassconj1}
If $a>1$ and $b>1$ are irrational numbers, then
$$\left| \log\frac{a}{b}-\sum_{n=1}^{\infty}\frac{a \{a^{-1} (n+1)\}-b\{b^{-1}(n+1)\}}{n(n+1)}\right|<\frac{\pi^2}{6}\text{.}$$
\end{conjecture}

\noindent However, numerical data suggest that the quantity in absolute value is less than $1$, i.e.
\begin{conjecture}
 \label{thbrassconj2}
If $a>1$ and $b>1$ are irrational numbers, then
$$\left| \log\frac{a}{b}-\sum_{n=1}^{\infty}\frac{a \{a^{-1} (n+1)\}-b\{b^{-1}(n+1)\}}{n(n+1)}\right|<1\text{.}$$
\end{conjecture}
Before proving Theorem \ref{thbrassthm1a}, we examine briefly some specific series similar to the one in that theorem, 
to gain some insight into the heuristics behind the convergence of this type of series.
 Let $a=\frac{1+\sqrt{5}}{2} $,  $b=a^2$,  $a_n=[
an]$ and $b_n=[ bn]$, as in Theorem \ref{thbrassthm1a}. What can we expect from the
 series $ \sum _{n=1}^{\infty} \left(\frac{\{a n \}}{a_n}-\frac{\{ b n\}}{b_n}\right)$?
Since $\{ x n\}$ is uniformly distributed for $x$ irrational, for large $m$ the $m^{th}$ partial
sum of the series behaves like
$\left(\frac{1}{a}-\frac{1}{b}\right)\frac{1}{2}\log m$, and so the series is
divergent. How about $ \sum _{n=1}^{\infty}\left(\frac{a_n}{a_{n+1}}-\frac{b_n}{b_{n+2}}\right)=\frac{1}{2}-\frac{2}{7}+\frac{2}{3}-\frac{5}{10}+\frac{3}{4}-\frac{7}{13}+...$?
Computing some of its partial sums gives: $s_{10}=1.22014275$,
$s_{50}=2.49049396$, $s_{100}=3.149858764$, $s_{500}=4.731312527$,
$s_{1000}=5.420950626$, $s_{10000}=7.720369134$,
$s_{50000}=9.329523382$. We thus expect that this series is divergent.
However, the series $ \sum _{n=1}^{\infty}
\left(\frac{a_n}{a_{n+1}}-\frac{b_n}{b_{n+1}}\right)=\left(\frac{1}{2}-\frac{2}{5}\right)+\left(\frac{2}{3}-\frac{5}{7}\right)+\left(\frac{3}{4}-\frac{7}{10}\right)+...,$
which involves only a minor change compared to the previous one, behaves ``better". Some of the partial sums for this one are
$s_{10}=-0.06921181263$, $s_{50}=-0.08157653315$,
$s_{100}=-0.08271773217$, $s_{500}=-0.08318152710$,
$s_{1000}=-0.08329909024$, $s_{10000}=-0.08340515936$,
$s_{50000}=-0.08342404240$. We start to suspect that we have
a convergent series. Similarly, if we invert the quotients and
consider instead the type of series given in the main theorem,
$ S:= \sum _{n=1}^{\infty} \left(\frac{a_{n+1}}{a_n}-\frac{b_{n+1}}{b_n}\right)=\left(3-\frac{5}{2}\right)+\left(\frac{4}{3}-\frac{7}{5}\right)+\left(\frac{3}{2}-\frac{10}{7}\right)+...,$
we see that its partial sums are $s_{100}=0.5463290032$, ...,$ s_{10000}=0.5475731159$, ... suggesting its convergence, as expected.
These types of series happen to have many interesting
properties, convergence being one of them. We rewrite the last series in a more revealing way.
 \begin{align*}
  S =&\frac{1}{2} -\frac{1}{15}+ \frac{1}{14} + \frac{1}{30}
-\frac{3}{104}+ \frac{1}{45}-\frac{2}{99}+ \frac{1}{60}+ \frac{2}{161}
-\frac{3}{208} +\frac{5}{476}+ \frac{5}{589} \\ &-\frac{4}{357}+
\frac{1}{132} -\frac{1}{104}+...
\end{align*}
 Notice that this series is not an alternating series. The occurrence of
positive and negative terms is not periodic. Moreover, the absolute values of the terms are not monotone.
Similarly, the positive terms are not monotone; e.g.,
$\frac{1}{60}> \frac{2}{161}$ but $\frac{2}{161}< \frac{1}{132}$ ( neither are the negative terms monotone, in general).
However, as can be appreciated when we rewrote $S$, the reason this type of series converges is that there is
enough cancelation. In fact, as it will be evident in the proof, for
the $m^{\text{th}}$ partial sum, we get $\log m$ cancelation. This is precisely
the order of magnitude of the heuristic estimate given above.

A special case here occurs when the irrationals $a$
and $b$ are such that $ 1\leq q<a<q+1$ and $ 1<b=r+\{a\}$, for $q,r
\in \mathbf{Z}$.  In Theorem \ref{thbrassthm1} we treat separately the 
case $a=\phi$ and $b=\phi^2$, where $\phi$ is the golden ratio, because in addition to being special in
the sense that we just described, it also has the special property
that it involves two complementary Beatty sequences. The proof
that we present here differs from the one for the general case. The proof for the particular case
works, with some minor modification, for any numbers $a,b>1$ such that $\frac{1}{a} +
\frac{1}{b}=1$, if $a$ and $b$ have the same fractional parts. Such examples can be created by setting $b=a+t$ for some 
fixed integer $t$, and solving the equation  $\frac{1}{a} +\frac{1}{b}=1$ for $a$.

\section{A Particular Case (Theorem \ref{thbrassthm1})}
\begin{thm}
\label{thbrassthm1}
Let $a=\frac{1+\sqrt{5}}{2}$ and $b=a^2=a+1$. Write $a_n:=[ an]$ and $b_n:=[
bn]$. Then the series
 $$\displaystyle \sum_{n=1}^\infty \left(\frac{a_{n+1}}{a_n}-\frac{b_{n+1}}{b_n}\right)$$
 converges.
\end{thm}
\begin{remark}
Note that in Theorem \ref{thbrassthm1}, $\{a\}=\{b\}=\frac{\sqrt{5}-1}{2}$, and that $1 < a < 2$ and
$2<b<3$. Also note that $\displaystyle{\frac{1}{a} +\frac{1}{b}
=1}$. Thus $A= ( a_n)_{n=1}^{\infty}$, and $B= (
b_n)_{n=1}^{\infty}$ are complementary Beatty sequences.
\end{remark}
The proof of this theorem can be divided into two parts. First, we
write the series as a sum over all the integers, instead of only those
of the form $[ a n ]$ and $[  b n ]$. Second,
we use summation by parts to complete the proof. In the first part of
the proof we will make use of equations 
\eqref{thbrass1}--\eqref{thbrass3} from \cite[Ch.\ 9]{as1}.
\begin{defin}
 \label{thbrassdef1} 
     For a real number $\alpha$, such that $0<\alpha<1$, we define the
characteristic function of $\alpha$ as
           \begin{equation}
               \label{thbrass1} 
                  f_{\alpha}(n)= [ \alpha (n+1) ]-
[ \alpha n ] \text{.}
           \end{equation}
\end{defin}
Clearly $f_{\alpha}(n)=1$ or $f_{\alpha}(n)=0$.  Since the sum
telescopes, we have
\begin{equation} 
\label{thbrass2}
\sum_{n=1}^{m} f_{\alpha}(n) =[ \alpha (m+1) ]\text{.}
\end{equation}
\begin{remark}
Notice that \eqref{thbrass2} gives the number of integers $n\leq m$  for
which $f_{\alpha}(n)=1\text{.}$
\end{remark}
\begin{defin} 
 \label{thbrassdef2}
    For $1<\beta \in \mathbf{R \setminus Q}$, define
    \begin{equation}
      \label{thbrass3} 
           g_{\beta}'(n)=
             \begin{cases}
              1 &\textit{if }     n=[ k\beta ],
\textit{for k }\hspace{1mm} \in \hspace{1mm}\mathbf{Z},\\
              0&      \hspace{3mm} \text{otherwise.}
             \end{cases}
    \end{equation}
\end{defin}

It is a well known fact \cite[Lemma\ 9.1.3]{as1} that for all integers $n$
\begin{equation} 
 \label{thbrass4}
 g_{\beta}' (n)= f_{1/\beta} (n)\text{.}
\end{equation}
\subsection{First Part of the Proof of Theorem \ref{thbrassthm1}}
We are going to prove the main parts of Theorem \ref{thbrassthm1} by a sequence
of lemmas. We start by rewriting $a_{n+1}$ and $b_{n+1}$\text{.}
\begin{lem}
\label{thbrasslem1}
\begin{equation}
a_{n+1}=\begin{cases}
 a_n + 2  &\text{if }     n \in \text{A},\\
 a_n + 1  &      \hspace{3mm} \text{otherwise}\text{.} \end{cases}
\end{equation}
\end{lem}
\begin{proof}
\hspace{5mm}
Note that \linespread{0.5}
\begin{align*}
 a_{n+1}-a_n =[ a(n+1)]-[ an]&=[ (1+\{a\})(n+1)]-[ (1+\{a\})n]\\
&=[ n+1+\{a\}(n+1)]-[ n+\{a\}n]\\
&=n+1-n +[ \{a\}(n+1)]-[ \{a\}n]\\
&= 1+f_{\{a\}}(n)\text{.}
\end{align*}
From \eqref{thbrass3} and the fact that $\frac{1}{\{a\}}=a$ we find that
$a_{n+1}=a_n+1+f_a'(n)$, and the result follows.
\end{proof}
Similarly, but using $2<b<3$ and $\{b\}=\{a\}$, we obtain the following lemma.
\begin{lem}
\label{thbrasslem2}
\begin{equation}
  b_{n+1}=\begin{cases}
  b_n + 3\text{,}  &\text{if }     n \in \text{A},\\
  b_n + 2  &      \hspace{3mm} \text{otherwise}\text{.} \end{cases}
\end{equation}
\end{lem}
Using these two lemmas, we immediately have:

\noindent\begin{minipage}{.45\linewidth}
 \centering
 \begin{equation}
\frac{a_{n+1}}{a_{n}}=\begin{cases}
 1+\frac{2}{a_n }\text{,}  &\textit{if }     \textit{n} \in \textit{A},\\
 1+\frac{1}{a_n}\text{,}  &      \hspace{3mm} \textit{otherwise} \text{.}\end{cases}
\end{equation}
\end{minipage}\hfill
\begin{minipage}{.45\linewidth}
 \centering
 \begin{equation}
  \frac{b_{n+1}}{b_{n}}=\begin{cases}
  1+\frac{3}{b_n}\text{,}  &\textit{if }     \textit{n} \in \textit{A},\\
  1+\frac{2}{b_n}\text{,}  &      \hspace{3mm} \textit{otherwise}\text{.} \end{cases}
\end{equation}
\end{minipage}
\vspace{1cm}

The sets $A$ and $B$ form a partition of the integers. Hence, after rewriting the $m^{th}$ 
partial sum of the series in Theorem \ref{thbrassthm1}
as
 $$ \displaystyle H:= \sum_{n=1}^m
\frac{a_{n+1}}{a_{n}}-\frac{b_{n+1}}{b_{n}}=
  \sum_{\substack{n \leq m \\
  n\in A }}
  \frac{2}{a_n }-\frac{3}{b_n}  +
\sum_{\substack{n \leq m \\
  n\in B }}
  \frac{1}{a_n }-\frac{2}{b_n}\text{,}
$$
\noindent we obtain Lemma \ref{thbrasslem3}.
\begin{lem}
\label{thbrasslem3}
Let $a$ and $b$ be given as in Theorem \ref{thbrassthm1} and let $H_1'$ and $H_2'$ be given by
 \begin{equation*}
 H_1'=  \sum_{\substack{n \leq m \\
          a_n=[ ka] \\
          k\in A }}
        \frac{2}{a_n } +
   \sum_{\substack{n \leq m \\
           a_n=[ ka] \\
           k\in B }}
   \frac{1}{a_n}\hspace{.5cm}
\end{equation*}
\noindent and
 \begin{equation}
     H_2':=\sum_{\substack{n \leq m \\
     b_n=[ kb ] \\
     k \in A }}
   \frac{3}{b_n}+
\sum_{\substack{n \leq m \\
     b_n=[ kb ] \\
     k \in B }}
   \frac{2}{b_n}\text{.}
\end{equation}
Then it follows that $H=H_1'-H_2'$.
\end{lem}
Next, set 
\begin{align}
\label{thbrass3a}
x=x(b,m)&= [ m b ] \\
y=y(a,m)&= [ m a] \label{thbrass3b}
\end{align}
 In the rest of this paper, unless otherwise specified,
when we write $x$ or $y$ we refer to $x(b,m)$ and $y(a,m)$ respectively. Note that
\begin{equation}
\label{thbrass4a}
 [ (x+1) b^{-1} ]= [ (mb+1-\{mb\}) b^{-1} ]=m+
[ (1-\{mb\}) b^{-1} ]=m\text{.} 
\end{equation}
The last equality holds because inside the bracket function we have
the product of two factors that are each less than one. Similarly
\begin{equation}
\label{thbrass5}
[ (y+1) a^{-1} ]=m\text{.} 
\end{equation}
Note that the sums $H_1'$ and $H_2'$ only run over numbers that are in $A$ and numbers that are in $B$,
respectively. As can be clearly seen, the $a_n$'s in the first sum of $H_1'$ come from multiplying by an element of
$A$ while the $a_n$'s in the second sum of $H_1'$ come from multiplying by an element of $B$. 
A similar arrangement can be observed in $H_2'$. 
For instance, the integer $10$ belongs to $B$, and it comes from multiplying by an element of $A$, since
$10=[ 4b]$ and $4\in A$. This observation suggests that in order to
write these two partial sums over all the integers, we need to introduce
indicator functions that tell us when a particular integer belongs to
the sequence $A$ or $B$, and when that integer comes from
multiplying by an element of $A$ or $B$. This is our immediate goal.
To do this we define the arithmetic functions given below in terms of
the characteristic functions of $\frac{1}{a}$ and $\frac{1}{b}$. The reasons for extending these two characteristic functions
in the form given below are as follows: As it can be clearly seen, the sum $H_1'$ runs over all elements of $A$ 
in the interval $[a_1,a_m]$, and $H_2'$ runs over numbers in $B$ in the interval $[b_1,b_m]$. We would like to combine 
these two sums into a single sum with upper limit given by $x$, which is the largest limit of summation in the two partial sums $H_1'$ and $H_2'$. 

\begin{defin}
 \label{thbrassdef3}
 Let $a$ and $b$ be the slope of two Beatty sequences $A$ and $B$, respectively, and for $\alpha$ irrational, $0<\alpha<1$, 
let  $f_{\alpha}$ be given by Definition \ref{thbrassdef1}. Also, let $x=x(b,m)$ and $y=y(a,m)$ be given as in \eqref{thbrass3a} and \eqref{thbrass3b}, 
respectively. Define $f$ and $g$ as follows\\
\noindent\begin{minipage}{.45\linewidth}
 \centering
 \begin{equation}
         \label{thbrass6}
             f(n)=
                 \begin{cases}
                     f_{\frac{1}{a}}(n)\textit{,}  & \hspace{5mm}\text{if }
 n\leq y\text{,}\\
                     0\text{,}&      \hspace{10mm} \text{otherwise.}
                  \end{cases}
     \end{equation}

\end{minipage}\hfill
\begin{minipage}{.45\linewidth}
 \centering
 \begin{equation}
         \label{thbrass6a}
             g(n)=
                 \begin{cases}
                     g_{\frac{1}{b}}(n)\textit{,}  & \hspace{5mm}\text{if }
 n\leq x\text{,}\\
                     0 \text{,}&      \hspace{10mm} \text{otherwise.}
                 \end{cases}
     \end{equation}

\end{minipage}
\end{defin}

\vspace{.5cm}
In this part of the proof we will make also use of the following
lemma from \cite[proof of Lemma\ 9.1.3]{as1}.
\begin{lem}
\label{thbrasslem4}
   Let $\alpha>1$ be any irrational  number. Then $n=[ \alpha
k]$ if and only if $ k=\left([
\frac{n}{\alpha}]+1\right)$ and $[ \frac{n+1}{\alpha}]=k$.
\end{lem}



In the following lemma we rewrite over all positive integers up to $[ mb]$, the $m^{\text{th}}$ partial sum associated to the 
series in Theorem \ref{thbrassthm1}. ( Note that the statement of the lemma is not symmetric in $x$ and $y$.)

\begin{lem}
\label{thbrasslem5} Let $a$ be the golden ratio ($a=\frac{1+\sqrt{5}}{2}$) and $b$ be the golden ratio squared ($b=a^2$). 
Let $x=[ m b ]$, and
let $f(n)$ and $g(n)$ be given by \eqref{thbrass6} and  \eqref{thbrass6a}, respectively. Then, $H=H_1-H_2$ where $H_1$ and $H_2$,
are given by
\begin{equation}
 \label{thbrass7}
     H_1= \sum_{n=1}^{x} \left(\frac{2f(n)f([\frac{n}{a}]+1)}{n} + \frac{f(n)g([\frac{n}{a} ]+1)}{n}\right)
\end{equation}
and
\begin{equation}
\label{thbrass8}
    H_2= \sum_{n=1}^{x} \left(\frac{3g(n)f([\frac{n}{a}]+1)}{n} + \frac{2g(n)g([\frac{n}{a} ]+2)}{n}\right)\text{.}
\end{equation}

\end{lem}

\begin{proof}
We will only provide the proof for $H_1$ because the proof for $H_2$ can be done in a similar manner. In this case, we can assume that $n\leq y$, since
$f(n)=0$ whenever $n>y$, by \eqref{thbrass5} and Definition \ref{thbrassdef3} . For the first term in $H_1$ it follows that $n=a_r \text{, for} \hspace{1mm} r \in A$ and $n\leq y$
 if and only if $f(n)f(r)=1$.
From Lemma \ref{thbrasslem4}, we find that $r=[ \frac{n}{a}
]+1$. Thus, the first summand in $H_1$ corresponds to the first sum in $H_1'$. Similarly
$n=a_r$, for $r \in B$ and $n\leq y$ if and only if $ f(n)g(r)=1$ and $r=[ \frac{n}{b}
]+1$ by Lemma \ref{thbrasslem4}. Hence the second summand of $H_1$ corresponds to the second sum in $H_1'$.
\end{proof}

Now that we have written each of the two partial sums in $H$ over all the integers up to
$x$, we are going to use partial summation to
estimate each of these four sums. This will complete the proof of
Theorem \ref{thbrassthm1}. But first we need to prove some auxiliary propositions, dealing with general partial
sums. We then apply them to the desired cases.

\section{Counting Special Elements of Beatty Sequences}

Applying summation by parts is one of the important ingredients in the
proof of the main theorem of this paper and its particular case. The terms of the sums
we work with are usually quotients involving characteristic
functions (say $f$ and $g$) of irrational numbers. They are of the
form $\displaystyle \frac{f(n)g([ \frac {n}{c}]+1)}{n}$ or
$\displaystyle \frac{f(n)}{n}$, or a combination of both. So, we start here by obtaining an expression for such types
of summatory functions. Then in Theorem \ref{thbrassthm3}, we find an asymptotic
formula for the $m^{\text{th}}$ partial sum whose terms are these types of
quotients.

 The first partial sum we find tells how many elements there are in a
Beatty sequence up to some given number $t$, and is presented in the following lemma. The first part of this lemma was proved earlier,
but we include it here for completion purpose.
\begin{lem}
\label{thbrasslem6a}
 Let $0<c<1$ be any irrational number, and let $f(n)= f_{c}(n)$.
 Then
\begin{equation}
\label{thbrass9}
 \sum_{n\leq t} f(n)=[ c \left([ t] +1\right)] \text{.}
\end{equation}
If $t=x=[ \frac{m}{c}]$, then
\begin{equation}
 \label{thbrass10}
\sum_{n\leq x} f(n)=m\text{.}
\end{equation}
\end{lem}

\begin{proof}
By \eqref{thbrass2},  the left hand side of \eqref{thbrass9} gives $[ c \left([ x] +1
\right)]$. And this gives $m$ by \eqref{thbrass4a}.
\end{proof}

We also need to count the elements of a given Beatty sequence that come from multiplying by
elements of another Beatty sequence. The lemma below involves sums encoding this type of selection of numbers. For $0<c<1$ and $0<d<1$, it counts the elements up to $t$
that are in the Beatty sequence with slope $1/c$ and come from multiplying by elements in the Beatty sequence with slope $1/d$.

\begin{lem}
\label{thbrasslem6}
 Let $0<d<1$ and $0<c<1$ be any irrational numbers, and let $f(n)= f_{c}(n)$ and $g(n)= g_{d}(n)$
 be given as in Definition \ref{thbrassdef3}.  Then 
 \begin{align}
\label{thbrass11}
            \sum_{n\leq t} f(n)g([ c\,n]+1) = [ d\left([ c \left([
t] +1 \right)] +1\right)]\text{.}
 \end{align}
\end{lem}
\begin{remark}
 As an illustration of this lemma, consider the following example. In the first $13$ elements of the Beatty sequence
 generated by $\phi$, where $\phi$ is the golden ratio,  i.e., 
$$1,3,4,6,8,9,11,12,14,16,17,19,21\text{,}$$ 
there are five numbers that come from multiplying by an element of the complementary sequence: $3,8,11,16,21$.
They are given by $[ 2  \phi ], [ 5  \phi ], [ 7  \phi
], [ 10  \phi ]$ and  $[ 13  \phi ] $. In other words, these numbers come from multiplying by elements in the Beatty  
complement, namely from multiplying by $2,5,7,10,13$, respectively. This quantity is encoded in the sum
$$\sum_{n\leq 21} f(n)g\left(\left[ \frac{1}{\phi}\,n\right]+1\right) = \left[ \frac{1}{\phi^2}\left([ \frac{1}{\phi} \left([
21] +1 \right)] +1\right)\right]=5\text{,}$$
\noindent and five is precisely the number one would expect according to the explanation given above.
\end{remark}
\begin{proof}
Let $C= \left( [ (1/c) n
]\right)_{n=1}^\infty$ and $D= \left( [ (1/d) n ]\right)_{n=1}^\infty$. Then, by \eqref{thbrass4} 
$f$ and $g$ are the indicator functions of the Beatty sequences $D$ and $C$.
Note that $f(n)g([ nc]+1)=1$ if and only if $n\in C$ and $([ nc ]+1) \in D$, by the definition of $f$ and $g$. 
In other words, $n$ is of the form $c_r$ with $r\in D$. By Lemma \ref{thbrasslem4}, we see that $r=[ (nc) ]+1$.
 Now by the remark following \eqref{thbrass2}, there are $s:=[ c \left([ t] +1\right)]$ integers up to $t$ for which $f(n)=1$. From these
integers, by the same remark, only $[ d(s+1) ]$ are of the form $[ (1/d) k ]$. Hence, the lemma holds.
\end{proof}

\begin{remark}
 \label{thbrassrem2}
Lemma \ref{thbrasslem6a} tells us that
$$[ c \left([ t] +1\right)]=\sum_{n\leq t} f(n)=\sum_{n\leq [ t ]} f(n)\text{,}$$
and thus, we can apply Lemma \ref{thbrasslem6a} to Lemma \ref{thbrasslem6} to see that 
$$\sum_{n\leq t} f(n)g([ c\,n]+1) = [ d([ c \left([ t] +1)] +1\right)] = \sum_{n\leq [ c \left([ t] +1 \right)] } g(n)\text{.}$$
Hence 
\begin{align}
\label{thbrass12ab}
\sum_{n\leq [ c \left([ t] +1 \right)] } g(n+k) &=  \sum_{k < n\leq [ c \left([ t] +1 \right)] +k} g(n)\\\nonumber 
                   &= \sum_{n\leq [ c \left([ t] +1 \right)] +k} g(n)-\sum_{n \leq k}g(n)\\\nonumber
 &=[ d\left([ c \left([ t] +1 \right)] +k +1\right)]  -[ d \left(k +1\right)]\text{.}
\end{align}
\end{remark}
We are now ready to generalize the previous lemma. Given two Beatty sequences $A$ and $B$, with characteristic sequences 
$f$ and $g$ respectively, and given a nonnegative integer $k$, we can count integers $n$ that are less than a real number 
$t$ and that also satisfy the following two conditions: first, $n$ is in $A$, and second, if $n=[ a_r ]$, 
then the integer $r+k$ is an element of $B$. Note also that this generalizes the previous lemma, because if $k=0$, then
this integer $n$ is an element of $A$ that is found by multiplying by an element of $B$. We find this type of partial sum next.

\begin{lem}
 \label{thbrasslem7}
 Let $0<d<1$ and $0<c<1$ be any irrational numbers, and let $f(n)= f_{c}(n)$ and $g(n)= g_{d}(n)$.
 Then, for any nonnegative integer $k$ we have 
 \begin{align}
\label{thbrass12}
            \sum_{n\leq t} f(n)g([ c\,(n)]+k+1) = [ d([ c \left([
t] +1 \right)]+ k+1)]- [ d( k+1)]\text{.}
 \end{align}
\end{lem} 

\begin{proof}
In light of \eqref{thbrass4}, $f(n)g([ c\,(n)]+1)$ equals $1$ if both $n=[ j /c ]$ 
and $j$ is in the Beatty sequence with slope $1/d$ (note that $j=[ c n ] +1$ by Lemma \ref{thbrasslem4}). 
Analogously, $f(n)g([ c\,(n)]+k+1)$ gives $1$ if $n=[ j /c ]$ and $j+k$ is in the Beatty 
sequence with slope $1/d$. Since there are $[ c([ t ]+1) ]$ integers up to $t$ in the 
Beatty sequence with slope $1/c$, we have
\begin{equation}
 \sum_{n\leq t} f(n)g([ c\,(n)]+k+1) = \sum_{j\leq [ c \left([
t] +1 \right)] } g(j+k)\text{.}
\end{equation}

\noindent By \eqref{thbrass12ab}, the last expression gives precisely the right hand side of \eqref{thbrass12}.
\end{proof}

\section{Some Useful Integrals and Sums}
For $c>1$ and $d>1$ real numbers, let $F_1(t)$ and $F_2(t)$ be given by 
\begin{equation}
\label{thbrass11d}
 F_1(t)=[ c ^{-1}\left([ t] +1 \right)]
\end{equation}
and
\begin{equation}
\label{thbrass12a}
 F_2(t)=[ d^{-1}([ c ^{-1}\left([ t] +1 \right)] +1)] \text{.}
\end{equation}
Now rewrite $F_1(t)$ and $F_2(t)$ using $[ x ]= x - \{ x\}$ repeatedly to obtain
\begin{equation}
 F_1(t)= \frac{1}{c} (t+1-\{t\}) - 
\left\{c ^{-1}\left([ t] +1 \right)\right\}\text{,}
\end{equation}
\noindent and 
\small  
\begin{align}
&F_2(t)= \\\nonumber
&(cd)^{-1} (t+1- \{ t \}) - d^{-1}\left\{c
^{-1}\left([ t] +1 \right)\right\} +d^{-1} -\left\{
d^{-1}\left(\left[ c^{-1}([ t ] +1)\right]+1\right)\right\}\text{.} 
\end{align}

\normalsize

\begin{lem}
\label{thbrasslem8}
 For any irrational number $c>1$, and a number $x$ given by $x=x(c,m)=[ c m]$ we have
 \begin{equation} 
   \label{thbrass20} 
       \displaystyle {\int_1^x \frac{[ c ^{-1}\left([ t] +1 \right)]}{t^2}dt = \frac{1}{c} 
(\log m + \log c + \gamma) -\sum_{n=1}^\infty \frac{\left\{c 
^{-1}\left(n+1 \right)\right\}}{n(n+1)}  + O\left(\frac{1}{m}\right)}\text{,} 
\end{equation}
where $\gamma$ is Euler's constant.
\end{lem}

\begin{proof}
Integrating $F_1$ we see that
\begin{align}
\int_1^x \frac{F_1(t)}{t^2}dt&= \int_1^x\frac{\frac{1}{c} (t+1-\{t\}) - \left\{c ^{-1}\left([ t] +1 \right)\right\}}{t^2}dt\\\nonumber
                &=\left. \frac{1}{c}\log x \right|_1^x + \int_1^x \frac{\frac{1}{c}(1-\{t\}) - \left\{c ^{-1}\left([ t] +1 \right)\right\}}{t^2}dt\text{.}
\end{align}

\noindent If we complete the tails of the integral we have
\begin{align}      
\int_1^x \frac{F_1(t)}{t^2}dt&=\left. \frac{1}{c}\log t \right|_1^x+ 
\frac{1}{c}\int_1^\infty \frac{1-\{t\}}{t^2}dt - \int_1^\infty \frac{ 
 \left\{c ^{-1}\left([ t] +1 
\right)\right\}}{t^2}dt \nonumber \\
& -\int_x^\infty \frac{\frac{1}{c}(1-\{t\}) - 
\left\{c ^{-1}\left([ t] +1 \right)\right\}}{t^2}dt\text{.}\nonumber
\end{align}

      The first integral in the right hand side equals $\frac{1}{c} \gamma$ (see \cite[Page 56]{apostol}). The third is 
$O\left( \frac{1}{x}\right)$, and given that $x=[ c m ]$,
this is really  $O\left( \frac{1}{m}\right)$.
      Also, for the same reason, $\log x = \log ( c m - \{ c m \}) =
\log (cm) + O(\frac{1}{m}) = \log c +\log m + O (\frac{1}{m})$. Hence we have
\begin{equation} 
         \label{thbrass21}
             \displaystyle {\int_1^x \frac{F_1(t)}{t^2}dt =
\frac{1}{c} (\log m + \log c + \gamma) + \int_1^\infty \frac{ 
\left\{c ^{-1}\left([ t] +1 \right)\right\}}{t^2}dt +
O\left(\frac{1}{m}\right)}\text{.}
\end{equation}

      We can write the infinite integral in \eqref{thbrass21} as an
infinite sum as follows:
\begin{align}      
\int_1^\infty \frac{ \left\{c ^{-1}\left([ t] +1 \right)\right\}}{t^2}dt&= \sum_{n=1}^\infty \int_0^1 \frac{\left\{c 
^{-1}\left(n +1 \right)\right\}}{(n+t)^2}dt \nonumber \\ 
 & = \sum_{n=1}^\infty \left\{c ^{-1}\left(n+1 \right)\right\} \int_0^1\frac{1}{(n+t)^2}dt\\\nonumber 
 &=\sum_{n=1}^\infty \frac{\left\{c ^{-1}\left(n+1 \right)\right\}}{n(n+1)}\text{.}\nonumber
\end{align}
      
\noindent Thus we see that
      \begin{equation*}\displaystyle {\int_1^x \frac{F_1(t)}{t^2}dt = \frac{1}{c} 
(\log m + \log c + \gamma) -\sum_{n=1}^\infty \frac{\left\{c 
^{-1}\left(n+1 \right)\right\}}{n(n+1)}  +
O\left(\frac{1}{m}\right)}\text{.}\qedhere
\end{equation*}
\end{proof}

\begin{lem}
\label{thbrasslem9}
   Let $r_n'$ be given by $r_n'=[ (1/r)n  ]$ for $r$ a real number.
 For any irrational numbers $c>1$ and $d>1$, and a number $x$ given by $x=x(c,m)=[ c m]$, we have
 \begin{align}
     \label{thbrass22}
\int_1^x &\frac{[ d^{-1}([ c ^{-1}\left([ t] +1 \right)] +k+1)]}{t^2}dt \\\nonumber
&= (cd)^{-1} (\log m + \log c + \gamma)+d^{-1}(k+1)\\\nonumber
& -\sum_{n=1}^\infty \frac{d^{-1}\left\{c ^{-1}\left(n+1
\right)\right\}+\{d^{-1}\left( c^{'}_{n+1}+k+1\right)\}}{n(n+1)}\\\nonumber & + O\left(\frac{1}{m}\right)\text{.}\nonumber
\end{align}
\end{lem}

\begin{proof}
Proceed in a similar manner to how we dealt with Lemma \ref{thbrasslem8}.
\end{proof}

\section{Applying Summation by Parts}
\begin{lem}
 \label{thbrasslem10}
Let $c>1$ be any irrational number, $f(n)$ be given in terms of $f_{\{ c\}}(n)$ as in Definition \ref{thbrassdef3}
 and  $c_n=[ c n ]$. If $C= (c_n)_{n=1}^\infty$ and $x=[ mc]$, then
\begin{equation}
  \label{thbrass23}
  \sum_{n=1}^x \frac{f(n)}{n} = \frac{1}{c}+ \frac{1}{c} (\log m + \log c + \gamma) 
                                  -\sum_{n=1}^\infty \frac{\left\{c
^{-1}\left(n+1 \right)\right\}}{n(n+1)}  + O\left(\frac{1}{m}\right)\text{.}
\end{equation}
\end{lem}

\begin{remark}
We can rewrite Lemma \ref{thbrasslem10} as
 $$\sum_{n=1}^x \frac{f(n)}{n}=
(c)^{-1} \log m + \alpha_ {c}+ O\left(\frac{1}{m}\right)\text{,} $$
 where $\alpha_ {c}$ is a constant.
\end{remark}
\begin{proof}
 We will apply summation by parts. Let $F(t)=\sum_{n\leq t}f(n)$. Then by Lemma \ref{thbrasslem6a}
and by \eqref{thbrass11d} we see that

\begin{equation}
F(t)=\sum_{n\leq t}f(n)=[ c \left([ t] +1\right)]=F_1(t)\text{.}
\end{equation}
Thus we have
\begin{align}
 \label{thbrass23a}
\sum_{n=1}^x\frac{f(n)}{n}&=\frac{F_1(x)}{x}+\int_1^x \frac{F_1(t)\,dt}{t^2}\text{.}\nonumber
\end{align}
\noindent From the definition of $x$, we see that 
$$\frac{F_1(x)}{x}=\frac1c + O\left(\frac1x \right)=\frac1c + O\left(\frac1m \right)\text{,}$$
The result follows by applying Lemma \ref{thbrasslem8} to the integral in \eqref{thbrass23a}.
\end{proof}

\begin{lem}
 \label{thbrasslem11}
Let $c>1$ be any irrational number, let $d^{-1}=\{c\}$,
 and let $x$ be given by $x=[ m c]$. Define $f$ and $g$ in terms of $f_{1/c}$ and $g_{1/d}$ 
respectively, as in Definition \ref{thbrassdef3}. Then
\begin{align}
 \sum_{n=1}^x &\frac{f(n)g([ \frac {n}{c}]+k+1)}{n} = \\\nonumber 
 &\frac{\{ c\}}{c}+(cd)^{-1} (\log m + \log c + \gamma)+\{d^{-1}(1+k)\}-\\
                               & \sum_{n=1}^\infty \frac{d^{-1}\left\{c ^{-1}\left(n+1
\right)\right\}+\{d^{-1}\left( c^{'}_{n+1}+k+1\right)\}}{n(n+1)}  + 
O\left(\frac{1}{m}\right)\text{.}\nonumber
\end{align}
\end{lem}

\begin{remark}
 Notice that we can rewrite Lemma \ref{thbrasslem11} as
 $$\sum_{n=1}^x \frac{f(n)g([ \frac {n}{c}]+1)}{n}=
(cd)^{-1} \log m + \alpha_ {c,d}+ O\left(\frac{1}{m}\right)\text{,} $$
 where $\alpha_ {c,d}$ is a constant.
\end{remark}

\begin{proof}
Let $\displaystyle F_3(t)=\sum_{n\leq t}f(n)g([ \frac {n}{c}]+k+1)$. We apply Lemma \ref{thbrasslem7} with $c$ and $d$
replaced by $c^{-1}$ and $d^{-1}$ respectively to see that
\begin{align}
 F_3(t)= [ d^{-1}([ c^{-1} \left([ t] +1 \right)]+ k+1)]- [ d^{-1}( k+1)]=F_2(t)-[ d^{-1}( k+1)]\text{.}
 \end{align}
The last equality follows from \eqref{thbrass12a} where $F_2(t)$ is defined. Upon applying summation by parts we obtain:
\small{\begin{align}
\label{thbrass12c}
\sum_{n=1}^x &\frac{f(n)g([ \frac {n}{c}]+k+1)}{n}= \\\nonumber &\frac{F_3(t)}{x}+\int_1^x\frac{F_3\,dt}{t^2}=\\\nonumber 
 &\frac{F_3(x)}{x}+\int_1^x\frac{[ d^{-1}([ c^{-1} \left([ t] +1 \right)]+ k+1)]}{t^2}dt -\int_1^x\frac{[ d^{-1}( k+1)]}{t^2}dt=\\\nonumber
& \frac{F_3(x)}{x}+\int_1^x\frac{F_2(t)}{t^2}\,dt- [ d^{-1}( k+1)]+ O\left(\frac1x \right)\text{.}
\end{align} }
Since $x=[ m c ]$, it follows that
\begin{align}
 \frac{F_3(x)}{x}=\frac{d^{-1}}{c}+O\left(\frac{1}{x} \right)=\frac{\{c\}}{c}+O\left(\frac{1}{m} \right)\text{.}
\end{align}
Now the integral involving $F_2$ in \eqref{thbrass12c} was evaluated in Lemma \ref{thbrasslem9}. Also we combine the term $d^{-1}(k+1)$ from 
Lemma \ref{thbrasslem9} with the expression $-[ d^{-1}(k+1)]$ in \eqref{thbrass12c} to get $\{d^{-1}(k+1)\}$. The result follows.
\end{proof}

\begin{thm}
\label{thbrassthm3}
Let $c>1$ be any irrational number and set $d= \{c \}^{-1}$ and
$q=c-\{c\}$. Let  $c_n=[ cn ]$, $C= (c_n)_{n=1}^\infty$, $d_n= [ d n ]$ and $D=(d_n)_{n=1}^\infty$.
Also let $ c_n' =  [ c^{-1} n ]$ and $x=x(c,m)=[ m
c]$.  Define $f(n)$ and $g(n)$ in terms of $f_{\frac{1}{c}}(n)$ and $g_{\frac{1}{d}}(n)$ 
respectively, as in Definition \ref{thbrassdef3}. Then 
\begin{align}
  \sum_{n=1}^x &\frac{q f(n)+f(n)g([ \frac {n}{c}]+k+1)}{n} = \\\nonumber
  &1+ \log m +\log   c+ \gamma +\left\lbrace\{c\}(k+1)\right\rbrace- \\ \nonumber 
 & \sum_{n=1}^\infty \frac{c\{ c^{-1} (n+1)\}+\{ \{c\}
(c^{'}_{n+1}+k+1)\}}{n(n+1)}+O\left(\frac{1}{m}\right)\text{,}\nonumber
\end{align}
where $\gamma$ is Euler's constant.
\end{thm}

\begin{remark}
 Theorem \ref{thbrassthm3} takes a partial sum and transforms it into a
main term, $\log m$, and an infinite sum. The partial sum involves only
quotients of elements that are in a given sequence, namely $C$. The
infinite sum involves fractional parts of elements of two sequences, 
$C^{-1}$ and $D^{-1}$ (where $R^{-1}$ is given by $R^{-1}=( [
r^{-1} n ])_{n=1}^\infty$). Since the infinite sum involves only
fractional parts in the numerator, and $n^2$ in the denominator, it is
convergent.
\end{remark}

\begin{proof} 
     If we put together Lemmas \ref{thbrasslem10} and \ref{thbrasslem11}, the partial sum in Theorem \ref{thbrassthm3} becomes

\begin{align}
         \label{thbrass25}
 &\sum_{n=1}^x \frac{q f(n)+f(n)g([ \frac {n}{c}]+k+1)}{n}=\\ \nonumber
             &\frac{q+\{c\}}{c}+\left( \frac{q}{c}+ (cd)^{-1}\right)(\log m + \log c +\gamma) + \left\lbrace \{c \} (k+1) \right\rbrace-\\
             &\sum_{n=1}^{\infty}\frac{q\left\lbrace c ^{-1}\left(n+1 \right)\right\rbrace+ d^{-1}\left\lbrace c ^{-1}\left(n+1 \right)\right\rbrace+\{d^{-1}\left(c^{'}_{n+1}+k+1\right)\}}{n(n+1)}  + O\left(\frac{1}{m}\right)\text{.}\nonumber 
\end{align}

     Notice that $\displaystyle{\frac{q}{c}+(cd)^{-1}=
\frac{c-\{c\}}{c}+\frac{\{c\}}{c}=1}$, and $\displaystyle{\frac{q+\{c\}}{c}=\frac{c}{c}=1}$. Also, since $q+d^{-1}=c$, we see that
$$q\left\lbrace c ^{-1}\left(n+1 \right)\right\rbrace+ d^{-1}\left\lbrace c ^{-1}\left(n+1 \right)\right\rbrace= c\left\lbrace c ^{-1}\left(n+1 \right)\right\rbrace\text{,}$$
and thus \eqref{thbrass25} becomes
     \begin{align}
         \label{thbrass26}
\sum_{n=1}^x &\frac{q f(n)+f(n)g([ \frac {n}{c}]+k+1)}{n}=\\\nonumber
      &1+ \log m +\log   c+ \gamma +\left\lbrace \{c \} (k+1) \right\rbrace -\\ \nonumber
             & \sum_{n=1}^\infty \frac{c\{ c^{-1} (n+1)\}-\{ d^{-1} (c_{n+1}'+k+1)\}}{n(n+1)}+O\left(\frac{1}{m}\right)\text{.}
     \end{align}

\end{proof}

 \section{Completion of Proof of Theorem \ref{thbrassthm1}}

\begin{proof}[Proof of Theorem \ref{thbrassthm1}]
 It suffices to prove that the partial sum $H=H_1-H_2$ (from Lemma
\ref{thbrasslem5}) is convergent, as $x$ approaches infinity. In order to do this, we apply
the remark following Lemma \ref{thbrasslem11} to $H_1$ and $H_2$. Since each sum is further divided into two partial 
sums we will apply it twice to each of these partial sums. For the first partial sum of $H_1$, the
parameters of the lemma are $c=a$, $d=a^{-1}$ and $ x=[
am]$. Also $f$ and $g$ of Lemma \ref{thbrasslem11} are both the same in the first partial sum of $H_1$, i.e., 
they are both defined in terms $f_{1/a}$. So the first term of $H_1$ gives
  \begin{align}
      \label{thbrass27}
\sum_{n=1}^x \frac{2f(n)f([ \frac{n}{a}]+1)}{n}&= (cd)^{-1} \log m + \alpha_ {c,d}+ O\left(\frac{1}{m}\right)\\\nonumber
                          &=2 (ad)^{-1} \log m + \alpha_ a+ O\left(\frac{1}{m}\right)\\\nonumber
                          &=\frac{2}{a^2}\log m + \alpha_ a+ O\left(\frac{1}{m}\right)\text{.}\nonumber
  \end{align}
For the second summand, $d=b^{-1}$ and thus $g$ is defined in terms of $g_{\frac{1}{b}}$. The other parameters remain the
same. Hence we have
  \begin{align}
      \label{thbrass28}
         \sum_{n=1}^x \frac{f(n)g([ \frac {n}{a}]+1)}{n}&= (cd)^{-1} \log m + \alpha_ {c,d}+ O\left(\frac{1}{m}\right)\\\nonumber
                                    &=(ad)^{-1} \log m + \alpha_{a,b}+O\left(\frac{1}{m}\right)\\\nonumber
                                    &=\frac{1}{ab} \log m + \alpha_ {a,b}+ O\left(\frac{1}{m}\right)\text{.}
  \end{align}
  Similarly, in $H_2$ we have $c=b$ and $d=a^{-1}$ for the first summand. Thus, the role of $f$ and $g$ in Lemma \ref{thbrasslem11} 
are switched in this first summand. Also, for the second summand of $H_2$, we take $c=b$ and $d=b^{-1}$. As a result,
in this second summand both $f$ and $g$ are the same, i.e., they are both defined in terms of $g_{\frac{1}{b}}$.
 In the end, the two summands of $H_2$ give 
  \begin{align}
      \label{thbrass29}
        H_2 &= \sum_{n=1}^{x} \left(\frac{3g(n)f([\frac{n}{b} ]+1)}{n} + \frac{2g(n)g([\frac{n}{b}]+2)}{n}\right)\\
             &=\left(3 (bd)^{-1}+ 2 (bd)^{-1})\right) \log m + \alpha_a+ O\left(\frac{1}{m}\right)\nonumber \\
             & = \left(\frac{3}{ab} + \frac{2}{b^2}\right) \log m + \alpha_{a,b}+ O\left(\frac{1}{m}\right)\text{.}\nonumber
  \end{align}
  When we put together \eqref{thbrass27}, \eqref{thbrass28} and
\eqref{thbrass29}, $H$ becomes
 \begin{equation}
      \label{thbrass30}
        H=H_1-H_2=\left(\frac{2}{a^2}+\frac{1}{ab}-\frac{3}{ab} -
\frac{2}{b^2}\right) \log m + \alpha_ {a,b}+ O\left(\frac{1}{m}\right)\text{.}
 \end{equation}
   Using $\frac{1}{a}+\frac{1}{b}=1$ and $a^2=b$, we see that
$$\frac{3}{ab}+\frac{2}{b^2}-\frac{2}{a^2}
-\frac{1}{ab}=\frac{2}{b}\left( \frac{1}{a}+\frac{1}{b}
\right)+\frac{1}{ab}-\frac{2}{a^2} -\frac{1}{ab}=
\frac{2}{b}-\frac{2}{a^2}=0\text{.}$$
Thus if we let $m$ tend to $\infty$ in \eqref{thbrass30}, we find that
$H=\alpha_{a,b}$.
\end{proof}

\section{Generalization} 
We now want to undertake the task of proving a more general version of
Theorem \ref{thbrassthm1}, one involving any irrational numbers $a,b>1$, and $a_{n+k}$ for any integer $k$.
This version will give an identity involving the type of series that occurs in Theorem
\ref{thbrassthm1}. Here Theorem \ref{thbrassthm2} is presented, and then Theorem \ref{thbrassthm1a} (in which the identity is found) is derived as a corollary.

\begin{thm} 
\label{thbrassthm2} 
Let $a>1$ and $b>1$ be irrational numbers. Let $a_n=[ an]$ 
and $b_n=[ bn]$. Also, let $a^{'}_n = [ a
^{-1}n]$  and $b^{'}_n = [ b ^{-1}n]$ Then, for any integer $k$ we have
 \begin{align*}
\sum_{n=1}^\infty \left( \frac{a_{n+k}}{a_n}-\frac{b_{n+k}}{b_n} \right)&=
 k\left( \log a -\log b\right) + \sum_{j=1}^{k}\left(\left\lbrace j\{ a \}\right\rbrace - \left\lbrace j\{ b \}\right\rbrace\right)-\\\nonumber                  
&\sum_{n=1}^{\infty}\left( \frac{ak\{ a^{-1} (n+1) \}-bk\{ b^{-1} (n+1) \}}{n(n+1)}\right)-\\\nonumber
&\sum_{n=1}^{\infty}\left( \frac{\sum_{j=1}^k\{ \{ a\} (a_{n+1}'+j) \}-\sum_{j=1}^k\{ \{ b\} (b_{n+1}'+j) \}}{n(n+1)}\right)\text{.}                  
\end{align*}
\end{thm}

By Weyl's theorem \cite[Page 1]{montgomery} for any irrational $\alpha$, and any fixed integer $q$, the sequences $\{\alpha n \}$ and 
$\{\alpha (q+n) \}$ are uniformly distributed, and thus 

$$\displaystyle \lim_{k\to \infty}\frac1k\sum_{j=1}^k \{ j \alpha \}=\lim_{k\to \infty}\frac1k\sum_{j=1}^k\{ \{ a\} (a_{n+1}'+j) \}=\frac12\text{.}$$
Upon dividing by $k$ and taking the limit as $k$ goes to infinity, we deduce the following corollary ( which is Theorem \ref{thbrassthm1a}).

\begin{corollary}
\label{thbrasscor3}
Let $a>1$ and $b>1$ be irrational numbers. Let $a_n=[ an]$ 
and $b_n=[ bn]$. Then we have:
 \begin{align}
\lim_{k\to \infty}\frac1k\sum_{n=1}^\infty & \left( \frac{a_{n+k}}{a_n}-\frac{b_{n+k}}{b_n}\right)+ \sum_{n=1}^{\infty}\frac{a\{ a^{-1} (n+1) \}-b\{ b^{-1} (n+1) \}}{n(n+1)}=\log \frac{a}{b}\text{.}
\end{align}

\end{corollary}



We proceed by rewriting the series in Theorem \ref{thbrassthm2} as a sum over all the
integers. When doing this, we will need to change the strategy used in
Theorem \ref{thbrassthm1}, since in general, we do not have the property
that $A$ and $B$ partition the integers (here $A=(a_n)_{n=1}^\i$ and
$B=(b_n)_{n=1}^\i$). When applying partial summation, we will use
Theorem \ref{thbrassthm3}, instead of the remarks following Lemmas \ref{thbrasslem10} and \ref{thbrasslem11}.

We begin with the left hand side of Theorem \ref{thbrassthm2}, namely with
\begin{equation}
 \label{thbrass31}
     \sum_{n=1}^\infty \left(\frac{a_{n+k}}{a_n}-\frac{b_{n+k}}{b_n}\right)\text{.}
\end{equation}
 Again we will give this proof in several lemmas. Since
$a=b$ gives zero trivially, throughout the proof we assume that $a<b$.

\begin{lem}
\label{thbrasslem12}
Let $a>1$ be any irrational number, and let $q$ be the unique positive integer such that $1\leq q<a<q+1$. Set $c^{-1}=\{a\}$ and
$f(n)=f_{\frac{1}{c}}(n)$. Then
\begin{equation}
\label{thbrass32}
a_{n+1}= a_n + q+ f(n)\text{.}
\end{equation}
\end{lem}

\begin{proof}
     Note that 
\begin{align}
a_{n+1}-a_n & =[ a(n+1)]-[
an]=[ (q+\{a\})(n+1)]-[ (q+\{a\})n] \nonumber \\
            & =[ qn+q+\{a\}(n+1)]-[ qn+\{a\}n] \nonumber \\
            &=qn+q-qn + [ \{a\}(n+1)]-[\{a\}n]= q+f_{\{a\}}(n)\text{,}\nonumber
\end{align}
 \noindent hence $a_{n+1}=a_n+q+f(n)\text{.}$
\end{proof}

Similarly, we can prove the following result:
\begin{lem}
\label{thbrasslem13}
Let $b>1$ be any irrational number, and let $r$ be the unique positive integer such that $1\leq r<b<r+1$. Set $d^{-1}=\{b\}$ and
$g(n)=g_{\frac{1}{d}}(n)$. Then
\begin{equation}
\label{thbrass32a}
b_{n+1}= b_n + r+ g(n)\text{,}
\end{equation}
where $g_{\frac{1}{d}}(n)$ is the characteristic function of $\frac{1}{d}$.
\end{lem}

\begin{remark}
 Lemmas \ref{thbrasslem12} and \ref{thbrasslem13} are written in a slightly
different ways than the corresponding lemmas in the proof of Theorem
\ref{thbrassthm1}. This will allow us to write the sum over all 
positive integers without requiring that the sequences $A$ and $B$
partition the integers. In Theorem \ref{thbrassthm1} we only needed two
characteristic functions, since in that case
$a=\{a\}^{-1}=\{b\}^{-1}$. In this theorem we will need four
characteristic functions. The changes in the rest of the proof of Theorem \ref{thbrass2} are consequences of these variations.
\end{remark}

\begin{lem}
\label{thbrasslem14}
Let $1\leq q<a<q+1$ be any irrational number. Set $c^{-1}=\{a\}$ and
$f(n)=f_{\frac{1}{c}}(n)$. Then for any integer $k$ we have
\begin{equation}
a_{n+k}= a_n + kq+ \sum_{j\leq k}f(n+j)\text{.}
\end{equation}
\end{lem}

\begin{proof}
From \eqref{thbrass31} we have
\begin{equation*}
a_{n+k}-a_n=\sum_{j\leq k} a_{n+j}-a_{n+j-1}= \sum_{j=0}^{k-1} q+f(n+j)=kq+\sum_{j=0}^{k-1}f(n+j)\text{.}\qedhere 
\end{equation*}

\end{proof}

We can now rewrite the summands in the first infinite series of Theorem
\ref{thbrassthm2} as
\begin{equation}
   \label{thbrass33}
       \displaystyle{\frac{a_{n+k}}{a_n}-\frac{b_{n+k}}{b_n}=
\frac{qk+\sum_{j=0}^{k-1}f(n+j)}{a_n}}-\frac{rk+\sum_{j=0}^{k-1}g(n+j)}{b_n}\text{.}
\end{equation}

Set $y=y(a,m)=[ am ]$ and $x=x(b,m)=[ bm ]$. As
we did in Theorem \ref{thbrassthm1}, we define the following two additional
arithmetic functions in  terms of the characteristic functions $h_\frac{1}{a}$ and $w_\frac{1}{b}$, 
respectively.

\begin{defin}
\label{thbrassdef4}
   Define\\
\noindent\begin{minipage}{.45\linewidth}
 \centering
 \begin{equation}
     \label{thbrass34}
             h(n)=
               \begin{cases}
                 h_{\frac{1}{c}}(n)\text{,}  & \hspace{5mm}\text{if }     n\leq y\text{,}\\
                 0\text{,}&      \hspace{10mm} \text{otherwise}\text{.}
               \end{cases}
 \end{equation}

\end{minipage}\hfill
\begin{minipage}{.45\linewidth}
 \centering
  \begin{equation}
     \label{thbrass35}
     w(n)=
         \begin{cases}
            w_{\frac{1}{d}}(n)\text{,}  & \hspace{5mm}\text{if }     n\leq x\text{,}\\
            0 \text{,}&      \hspace{10mm} \text{otherwise.}
         \end{cases}
 \end{equation}

\end{minipage}

\end{defin}
\hspace{.5cm}

Now we use \eqref{thbrass33} and Definition \ref{thbrassdef4} to
rewrite the first infinite sum in Theorem \ref{thbrassthm2}, in the
following way.

\begin{align}
 \label{thbrass36}
\sum_{n\leq m}\left(\frac{a_{n+k}}{a_n}-\frac{b_{n+k}}{b_n}\right)   & =\sum_{n \leq x} \frac{qk+\sum_{j=0}^{k-1}f(n+j)}{a_n}-\sum_{n \leq y}\frac{rk+\sum_{j=0}^{k-1}g(n+j)}{b_n}\\
                                                        & = \sum_{n \leq x} \frac{qkh(a_n)+\sum_{j=0}^{k-1}h(a_n)f(n+j)}{a_n}\\\nonumber
       &- \sum_{n \leq y}\frac{rk w(b_n)+\sum_{j=0}^{k-1}w(b_n)g(n+j)}{b_n}\text{.} \notag
\end{align} 

Using \eqref{thbrass4} we inmediately obtain the following lemma.
\begin{lem}
\label{thbrasslem14a}
For any two Beatty sequences $a_n$ with slope $a$ and $b_n$ with slope $b$,  let $f$ and $g$, be the characteristic sequences of
 $\{a\}$ and $\{b\}$, respectively. Also let $h$ and $w$  be given as in Definition \ref{thbrassdef3} in terms of $h_{1/a}$ and $w_{1/b}$, 
respectively. Then
\begin{align}
 \label{thbrass37a}
\sum_{n\leq m}\left(\frac{a_{n+k}}{a_n}-\frac{b_{n+k}}{b_n}\right)   &=  \sum_{n \leq x} \frac{qkh(n)+\sum_{j=0}^{k-1}h(n)f([ \frac{n}{a} ]
+j+1)}{n}\\\nonumber & - \sum_{n \leq y}\frac{rkw(n)+\sum_{j=0}^{k-1}w(n)g([ \frac{n}{a} ]+j+1)}{n}\text{,}
\end{align}
where $q=[ a ]$ and $r=[ b ]$.
\end{lem}

\begin{thm}
\label{thbrassthm4}
Let $c>1$ be any irrational number and set $d= \{c \}^{-1}$ and
$q=c-\{c\}$. Let  $c_n=[ c n ]$, $C= (  c_n)_{n=1}^\infty$, $d=  [ dn ]$ and $D=(d_n)_{n=1}^\infty$.
Also let $ c^{'}_n = [ c^{-1} n ]$ and for any integer $m$, let $x=x(m,c)=[ m
c]$.  Let $f(n)$ and $g(n)$ be given as in Definition \ref{thbrassdef3} in terms of $f_{\frac{1}{c}}(n)$, and $g_{\frac{1}{d}}(n)$. Then 
\begin{align*}
  \sum_{n=1}^x &\frac{kq f(n)+\sum_{j=1}^k f(n)g([ \frac {n}{c}]+j+1)}{n} = \\\nonumber 
  &k\left(1+ \log m +\log   c+ \gamma \right) +\sum_{j=1}^k\{\{c\}j\}- \\ \nonumber 
 & k\sum_{n=1}^\infty \frac{c\{ c^{-1} (n+1)\}+\{ \{c\}(c^{'}_{n+1}+k+1)\}}{n(n+1)}+\\\nonumber
&O\left(\frac{k}{m}\right)\text{.}\nonumber
\end{align*}
\end{thm}

\begin{proof}
 Note that
\small{\begin{align}
 \sum_{n=1}^x \frac{kq f(n)+\sum_{j=1}^k f(n)g([ \frac {n}{c}]+j+1)}{n} = &\sum_{n=1}^x \frac{\sum_{j=1}^k \left( q f(n)+f(n)g([ \frac {n}{c}]+j+1)\right)}{n}\\\nonumber
  =&\sum_{j=1}^k \sum_{n=1}^x \frac{ q f(n)+f(n)g([ \frac {n}{c}]+j+1)}{n}\text{.}\nonumber
\end{align}}
Now apply Theorem \ref{thbrassthm3} to the inner sum to get
\begin{align}
\sum_{j=1}^k &\sum_{n=1}^x \frac{q f(n)+f(n)g([ \frac {n}{c}]+j+1)}{n} = \\\nonumber
&\sum_{j=1}^{k} \left(O\left(\frac{1}{m}\right)+ 1+ \log m +\log   c+ \gamma +\left\lbrace\{c\}(k+1)\right\rbrace -\right.\\\nonumber
&\left. \sum_{n=1}^{\infty} \frac{c\{ c^{-1} (n+1)\}+\{ \{c\}(c^{'}_{n+1}+k+1)\}}{n(n+1)}\right)\text{,}\nonumber
\end{align}
and then sum from $j=1$ to $k$ to complete the proof.
\end{proof}

We can now give the proof of Theorem \ref{thbrassthm2}.

\section{Proof of Theorem \ref{thbrassthm2}}
From Lemma \ref{thbrasslem14a} we know that
\begin{align}
 \label{thbrass37b}
\sum_{n\leq m}\left(\frac{a_{n+k}}{a_n}-\frac{b_{n+k}}{b_n}\right)   &=  \sum_{n \leq x} \frac{qkh(n)+\sum_{j=0}^{k-1}h(n)f([ \frac{n}{a} ]
+j+1)}{n}\\\nonumber & - \sum_{n \leq y}\frac{rkw(n)+\sum_{j=0}^{k-1}w(n)g([ \frac{n}{b} ]+j+1)}{n}\text{.}
\end{align}
\noindent By Theorem \ref{thbrassthm4}  we see that
\begin{align}
  \label{thbrass24a}
  \sum_{n=1}^x \frac{kq h(n)+\sum_{j=1}^k h(n)f([ \frac {n}{a}]+j+1)}{n} &= k\left(1+ \log m +\log   a+ \gamma \right) \\ \nonumber 
 & +\sum_{j=1}^k\{\{a\}j\} -\sum_{n=1}^\infty \frac{ak\{ a^{-1} (n+1)\}}{n(n+1)}\\\nonumber
 &- \sum_{n=1}^\infty \frac{\sum_{j=1}^k\{ \{a\}(a^{'}_{n+1}+j)\}}{n(n+1)}\\\nonumber
 & +O\left(\frac{k}{m}\right)\text{.}
\end{align} 
\noindent Also 
\begin{align}
  \label{thbrass24b}
  \sum_{n=1}^x \frac{kq w(n)+\sum_{j=1}^k w(n)g([ \frac {n}{b}]+j+1)}{n} &= k\left(1+ \log m +\log   b+ \gamma \right)\\ \nonumber 
 &+ \sum_{j=1}^k\{\{b\}j\} -\sum_{n=1}^\infty \frac{kb\{ b^{-1} (n+1)\}}{n(n+1)}\\\nonumber
&  -\sum_{n=1}^\infty \frac{\sum_{j=1}^k\{ \{b\}(b^{'}_{n+1}+j)\}}{n(n+1)}\\\nonumber
 &+O\left(\frac{k}{m}\right)\text{.}
\end{align}
We subtract \eqref{thbrass24b} from \eqref{thbrass24a} to obtain
 \begin{align*}
\sum_{n=1}^\infty \left( \frac{a_{n+k}}{a_n}-\frac{b_{n+k}}{b_n} \right)&=
 k\left( \log a -\log b\right) + \sum_{j=1}^{k}\left(\left\lbrace j\{ a \}\right\rbrace - \left\lbrace j\{ b \}\right\rbrace\right)-\\\nonumber                  
& \sum_{n=1}^{\infty}\left( \frac{ak\{ a^{-1} (n+1) \}-bk\{ b^{-1} (n+1) \}}{n(n+1)}\right)-\\\nonumber
&\sum_{n=1}^{\infty}\left( \frac{\sum_{j=1}^k\{ \{ a\} (a^{'}_{n+1}+j) \}-\sum_{j=1}^k\{ \{ b\} (b^{'}_{n+1}+j) \}}{n(n+1)}\right)\text{.}                  
\end{align*}
This completes the proof.

\begin{remark}
 In the last steps of this proof we see the $\log m$ cancellation mentioned in the discussion in the intruductory section
\end{remark}

\section{Definition of Sturmian Sequences from Beatty Ratios and Final Remarks \label{thbrasssec6}}
The identity of Theorem \ref{thbrassthm2}

\begin{align*}
\sum_{n=1}^\infty \left( \frac{a_{n+k}}{a_n}-\frac{b_{n+k}}{b_n} \right)&=
 k\left( \log a -\log b\right) + \sum_{j=1}^k\left(\left\lbrace j\{ a \}\right\rbrace - \left\lbrace j\{ b \}\right\rbrace\right)-\\\nonumber                  
& \sum_{n=1}^{\infty}\left( \frac{ak\{ a^{-1} (n+1) \}-bk\{ b^{-1} (n+1) \}}{n(n+1)}\right)-\\\nonumber
&\sum_{n=1}^{\infty}\left( \frac{\sum_{j=1}^k\{ \{ a\} (a^{'}_{n+1}+j) \}-\sum_{j=1}^k\{ \{ b\} (b^{'}_{n+1}+j) \}}{n(n+1)}\right)                  
\end{align*}
can be written as 
$$\sum_{n=1}^\infty \left(\frac{a_{n+k}}{a_n}-\frac{b_{n+k}}{b_n}\right)+ P(a,b) = k(\log a- \log b)\text{,} $$

\noindent where the constant $P(a,b)$ representes the obvious missing expressions. This identity brings to mind the theorem of 
Frullani \cite{fr1} which says that, under certain conditions on $f$, 

$$ \int_0^\infty \frac{f(ax) -f(bx)}{x}dx=A(\log a - \log b)\text{.}$$

We conclude this paper by providing a new way to define Sturmian sequences in terms of differences of quotients of the form 
$\displaystyle \frac{a_{n+1}}{a_{n}}- \frac{b_{n+1}}{b_{n}}$. Sturmian sequences can also be defined in terms of differences of 
reciprocals of these quotients and other forms of difference quotients. This is in the spirit of
\cite{as1}, \cite{kimberling} and \cite{lothaire} in which the authors present several
 equivalent ways in which Beatty and Sturmian sequences can be generated.

\begin{thm}

 Let  $a$ and $b$ be irrational numbers such that $1<a<b$, and $\theta= \{ a \}=\{b\}$. Define 

\begin{equation*}
  h(n)= \frac{a_{n+1}}{a_{n}}- \frac{b_{n+1}}{b_{n}}\text{,}
\end{equation*}
let
\begin{equation*}
  f(n)=\left\{\begin{array}{ll}
		1\text{,} & h(n)>0\text{,}\\
		0\text{,} & \text{otherwise,}
	\end{array}\right.
\end{equation*}

\noindent and set $\textbf{f}:=f(1)f(2)f(3)...f(n)...$. Also, let $g:=g_{\theta}$ be the characteristic (Sturmian) sequence with 
slope $\theta$ defined by equation \eqref{thbrass3}, and set $\textbf{g}:=g(1)g(2)g(3)...g(n)...$. Then

  \begin{center} \textbf{f}=\textbf{g}\text{.} \end{center}

\end{thm}

\begin{remark}
 By \eqref{thbrass4} we see that $h(n)>0$ if and only if $n$ is in the Beatty sequence $B$ with slope $1/\theta$. Thus, given 
any Beatty sequence, we can define an infinite word \textbf{g} with a $1$ or a $0$ in each position, and
 such that \textbf{g} contains a $1$ in the $n^{\text{nth}}$ position if and only if $n\in B$.
\end{remark}

\begin{proof}
 Suppose $b=r+ \theta$ and $a=q+\theta$.  Write 
$b_n=[ (r+\theta)n]=rn+[ \theta n ]=: rn+\theta_n$, and similarly write $a_n=qn+\theta_n$.
One can use \eqref{thbrass32} and \eqref{thbrass32a} to write $h(n)$ as
\begin{align}
 \label{thbrass40a}
h(n)&=\frac{a_{n+1}}{a_n}-\frac{b_{n+1}}{b_n}=\frac{a_{n}+q+g(n)}{a_n}-\frac{b_{n}+r+g(n)}{b_n}\\\nonumber
    &=\frac{\theta_n(q-r)+ng(n)(r-q)}{a_nb_n}=\frac{(ng(n)-\theta_n)(r-q)}{a_nb_n}\text{.}
\end{align}

\noindent Since  $\theta_n<n$ (because $\theta<1$) and $r>q$, it follows that $g(n)=1$ if and only if $h(n)>0$. Thus, the theorem holds.
\end{proof}

\section{Acknowledgements}
The author thanks Prof. Kenneth B. Stolarsky for a number of helpful discussions.

\bibliographystyle{apa}

\end{document}